\documentclass[12pt]{article}

\usepackage{setspace}
\usepackage{latexsym}
\usepackage{amsmath}
\usepackage{amsfonts}
\usepackage{amssymb}
\usepackage{enumerate}

\usepackage{graphics}
\usepackage[]{graphicx}

\include{amsthm,amsmath}
\include{XThesis}
\include{amssymb}
\include{mathrsfs}

\begin{document}

%\newfont{\rams}{msbm10 scaled\magstep1}
\newcommand{\rea}{\mbox{\rams \symbol{'122}}}
\newcommand{\rat}{\mbox{\rams \symbol{'121}}}
\newcommand{\nat}{\mbox{\rams \symbol{'116}}}
\newcommand{\com}{\mbox{\rams \symbol{'103}}}
\newcommand{\inte}{\mbox{\rams \symbol{'132}}}
\newtheorem{theorem}{Theorem}[section]
\newtheorem{definition}{Definition}[section]
\newtheorem{remark}{Remark}[section]
\newtheorem{example}{Example}[section]
\newtheorem{lemma} {Lemma}[section]
\newtheorem{corollary}{Corollary}[section]
\newtheorem{proposition}{Proposition}[section]
\newtheorem{case}{Caso}
\newenvironment{proof}{\par\noindent{\bf
Proof\/}}{\hfill $\Box$ \vskip .3cm}

\title{On Lagrange duality theory for dynamics vaccination games}

\author{Annamaria Barbagallo \\
Department of Mathematics and Applications ``R. Caccioppoli'',\\ University of Naples Federico II,\\
via Cinthia - 80126 Naples, Italy \\
e-mail: annamaria.barbagallo@unina.it \\
Maria Alessandra Ragusa \\
Department of Mathematics and Computer
Science, \\ University of
Catania, \\ Viale A. Doria, 6 - 95125 Catania - Italy \\
e-mail: maragusa@dmi.unict.it}

\date{}
\maketitle

\begin{abstract}
The authors study an infinite dimensional duality theory
finalized to obtain the existence of a strong duality between a
convex optimization problem connected with the management of
vaccinations and its Lagrange dual. Specifically, the authors show
the solvability of a dual problem using as basic tool an hypothesis
known as {\it Assumption S}. Roughly speaking, it requires to show
that a particular limit is nonnegative. This technique improves the previous strong duality results that need the nonemptyness of the interior of the convex ordering cone.
The authors use the duality theory to analyze the dynamic vaccination game
in order to obtain the existence of the Lagrange multipliers related to the
problem and to better comprehend the meaning of the problem.

\bigskip


\end{abstract}

\newcommand{\R}{{\Bbb R}}
\newcommand{\Z}{{\Bbb Z}}

\section{Introduction}

The main goal of the paper is to present general results for the infinite
dimensional duality theory and to utilize them for analyzing dynamic
vaccination games.

%The infinite-dimensional duality theory will be applied in order to
%obtain the characterization of the dynamic vaccination game
%conditions by means of the Lagrange multipliers.

The duality theory (see
\cite{DGI,DGM,MauRaciti}) firstly was defined to investigate the problem
of finding, in the infinite dimensional framework, the Lagrange
multipliers connected to an optimization problem or to a
variational inequality subject to possibly nonlinear constraints
when the interior of the ordering cone is
empty. This research line, combined with a generalized constraint
qualification assumption, the well known {\em Assumption S},
ensures the existence of the strong duality between a convex
optimization problem and its Lagrange dual. The employment of the
quasi-relative interior, firstly presented by Borwein and Lewis
\cite{BorwLew}, and the concepts of tangent and normal, cone permit to
go beyond the mentioned difficulty that in a lot of problems the interior of the constraint set is empty. For example, the reader can think about network equilibrium problems, the obstacle problem, the
elastic-plastic torsion problem (see e.g.
\cite{BarbaMau,DG,DonMauMilVit,DGI,DGM,MauRaciti,BarbaDiVi,BarbaDiViPia,BarbaDanMau,DonMil}),
which all use positive cones of Lebesgue or Sobolev spaces. It is useful to point out that {\em Assumption S} is also a necessary condition to ensure the
strong duality (see \cite{BotCseMoldo}). %Also,

In the present work, we focus our attention to archive the dual formulation of the dynamic
vaccination game. In the last decade, several authors devoted their interest in the application of the game
theory to vaccinating behavior under voluntary policies for
human diseases \cite{[4],[5]}, for instance measles, mumps, pertussis and rubella \cite{[2]}. In the above mentioned notes a homogeneous population where all individuals share
the same perception of risk is assumed.
Nevertheless, in real populations, risk perception can fluctuate a lot if we treat different social groups, or units, \cite{[20],[29]}. The paper \cite{[13]} deals with the dynamics of vaccinating behavior in a population
divided into social units, each one having a different perceived risk
of infection and vaccination, under a voluntary policy, via
projected dynamical systems and variational inequalities. %The present paper is
%using some of the setup of \cite{[13]}, but is answering a different
%question.

In [24, 11, 8], it is underlined that, whether or not an individual decides to vaccinate,
the perceived probability of their becoming infected rests on the level of disease  prevalence.
Disease prevalence depends on the vaccine coverage in the population (see [1]).
If vaccination is voluntary, the desease prevalence is the combined effect of the vaccination decisions of other people.
As a consequence, the individuals in a given
population are in practice engaged in a crucial interaction (a
'game') with one another, utilized transmission dynamics. In particular, we stress the importance of the evolution of group's equilibrium vaccinating
strategies in a community divided into social units, each one
having a special perceived risk of infection and vaccination.
Let us consider an infectious disease, for  which vaccination occurs shortly after birth, where parents voluntary choose to vaccinate their children, and in which individuals can be either susceptible, infectious, or immune. These are known as Susceptible-Infectious-Recovered (SIR) models, successfully tested in several epidemiology cases (see for instance \cite{[2]}).
%Let us suppose an infectious disease, which vaccination occurs
%only shortly after birth for, where parents voluntary choose to vaccinate their children, and in which individuals can
%be either susceptible, infectious, or immune. These are
%known as Susceptible-Infectious-Recovered (SIR) models, successfully tested in several epidemiology cases (see for instance \cite{[2]}).

In \cite{[13]} and in
\cite{Coj2008}, population biology games have been firstly studied via
variational inequalities. In \cite{BarbaCoj3} the authors analyze the evolution of the equilibrium strategies of each unit for
a time interval $[0,T]$ and considered $p(t)$ a vaccination
coverage function reflecting a vaccine scare happening in a
population over $[0,T]$. An help to solve this
argument is provided by generalized Nash games. Thanks to the theory of variational
inequalities, Barbagallo and Cojocaru present a method to obtain and compute a solution to this model. At last, in \cite{CojGree},
hybrid system based on a continuous dynamics described by a
projected system and on a finite set of exogenous event times is illustrated. Moreover, such  hybrid system finds an application to track evolution strategies.

We stress that several problems arising from the theory of
population dynamics and other physics phenomena are studied, not only
for partial differential equations (see
e.g. \cite{%R1,R2,R3,
R4,R5%,RT
}), but also by means of evolutionary
variational inequalities (see for instance \cite{Fragnelli1,
Fragnelli2}). %Moreover, we refer to the papers devoted to the
%projected dynamical systems (see \cite{CojDanNag,GIP}), to the
%weighted traffic equilibrium problem (see \cite{GP,B2009NA}) and to
%the regularity results for solutions to evolutionary variational
%inequality (see \cite{BP2011}).

The paper is organized as follows. In Section \ref{S:PCR} we recall
the infinite dimensional duality theory. In Section \ref{S:APP} we
present the dynamic vaccination game and show how Lagrange and
duality theory can be applied to this model.
After that, we are able to prove the existence of the Lagrange multipliers associated to the problem,
which is  very useful to study  the behaviour of the vaccination game.

\section{Lagrange duality}\label{S:PCR}

Let us fix $X$ a topological set, $Y$ a real normed space ordered
by a convex cone $ C $ with dual space $Y^*$, $Z$  a real
normed space and $Z^*$ its dual space. Let us consider $S$ a convex subset of
$X$, $f:S \to \R $ a functional,  $g:S \to Y $ a
mapping and $h:S \to Z $ an affine-linear mapping. Moreover, let us set
\begin{equation}\label{1}
K \,=\, \{ x \in S\,:\, g(x)\in -\,C,\, h(x)\,=\,\theta_Z \},
\end{equation}
where $\theta_Z $ is the zero element in the space $Z. $

Let us analyze the optimization problem to find $x_0 \in K $
such that
\begin{equation}\label{2}
f(x_0)\,=\, \min_{x \in K } f(x).
\end{equation}
Let us remark that its Lagrange dual problem is
\begin{equation}\label{3}
\max_{u \in C^*\!, v \in Z^*} \,\, \inf_{x \in S} \,\{f(x)\,+\,
\langle u, g(x)\rangle\,+\, \langle  v, h(x) \rangle \}
\end{equation}
where we indicate the usual dual cone of $C$ by
\begin{equation}\label{4}
C^*\,=\, \{ u \in Y^*\,:\, \langle u,y \rangle \geq 0, \forall y \in
C \}.
\end{equation}
Let us now recall the definition of {\em Assumption S}.
Fixed three functions  $f, g, h$ as before and  $K $ as in
(\ref{1}), we declare that {\em Assumption S} is verified at a point
$x_0 \in K $ if and only if
\begin{equation}\label{7}
T_{\tilde M} (0,\theta_Y, \theta_Z)\,\cap (\,]- \infty,
0[\times\,\{\theta_Y \}\times\,\{\theta_Z \}\,)\,=\,\emptyset,
\end{equation}
being
\begin{equation}\label{8}
{\tilde M}= \{ (f(x)-f(x_0)+ \alpha, g(x)+y, h(x)): \ x \in
S\backslash K, \ \alpha \geq 0, \ y \in C  \}.
\end{equation}
Throughout the exposition we indicate by $T_C(x)$ the tangent cone to $C$ at $x$ defined as
\begin{eqnarray}\label{S}
T_C(x)= \{ g \in X: && g\,=\,\lim_{n \to \infty } \lambda_n\, (
x_n\,-\, x), \lambda_n \in \R^+ , \forall n \in N,
\\ && x_n \in C, \forall  n \in N, \lim_{n \to \infty} x_n\,=\,x\}.
\nonumber
\end{eqnarray}

We are now enable to report the primary development on strong duality theory.

\begin{theorem}\label{thm1}(see \cite{DGM})
Let us assume $f:S \to \R, $ $g:S \to Y $ two convex functions,
 $h:S \to Z $ an affine-linear map and  {\em Assumption S} is fulfilled in the optimal solution $x_0 \in K $
to (\ref{2}), thus, problem (\ref{3}) is resolvable and, whenever ${\bar u}
\in C^*, $ $ {\bar v} \in Z^* $ are the optimal solutions to
(\ref{3}), we purchase
\begin{equation}\label{9}
\langle  {\bar u}, g(x_0)\rangle\,=\,0
\end{equation}
 and the optimal values of (\ref{3}) coincides, specifically
\begin{equation}\label{10}
f(x_0)= \max_{u \in C^*, v \in Z^*} \inf_{x \in S} \left(  f (x) +
\langle { u},g(x)\rangle + \langle { v},h(x)\rangle \right).
\end{equation}
\end{theorem}

{\em Assumption S} is archivable to accomplish  that a suitable limit is
larger or equals than zero.  The present requirement is necessary in order to gain strong duality
(see \cite{BotCseMoldo}).

It is time to formalize the coming Lagrange functional
\begin{%displaymath
equation}\label{Lagrange_functional}
\mathcal{L} (x,u,v) = f(x) + \langle { u},g(x)\rangle + \langle {
v},h(x)\rangle, \quad \forall x \in S, \ u \in C^*, \ v \in Z^*.
\end{%displaymath
equation}
Consequently, $(\ref{10})$ can be updated as follows
\begin{displaymath}
f(x_0) = \max_{u \in C^*, v \in Z^*} \inf_{x \in S} \mathcal{L}
(x,u,v).
\end{displaymath}
It develops, from what just considered, the below property

%{\sl Proof.}
%
%Let us now state in the next theorem an useful consequence of the strong duality in the usual relationship between a saddle point of the Lagrange functional
%\begin{equation}\label{11}
%\!\!\!\!\!\!\!\!\!\!\!\!\!\!\!\!\!\!\!\!\!\!\!{\cal L} (x,u,v)= f(x) + \langle u, g(x) \rangle+ \langle v, h(x) \rangle , \,\,\forall x \in S, \forall u \in C^*, \forall v \in Z^*,
%\end{equation}
%and the solutions of \eqref{2} and \eqref{3}.
%
\begin{theorem}\label{thm2}(see \cite{DGI}).
Let us assume that the same the assumptions of Theorem \ref{thm1} are satisfied. Because of these, $x_0
\in K $ is an optimal solution of (\ref{2}) if and only if
there exist ${\bar u} \in C^* $ and ${\bar v} \in Z^* $ such that
$(x_0, {\bar u}, {\bar v}) $ is a saddle point of \eqref{Lagrange_functional}, equivalently
\begin{equation}\label{12}
%\mathcal
\mathcal{L} (x_0,u,v)\leq  \mathcal{L} (x_0,\bar u,\bar v) \leq
\mathcal{L} (x,\bar u,\bar v), \,\,\forall x \in S, \forall u \in
C^*, \forall v \in Z^*.
\end{equation}
In addition
\begin{equation}\label{13}
\langle  {\bar u}, g(x_0) \rangle\,=\,0.
\end{equation}
\end{theorem}

\section{Applications to dynamic vaccination games}\label{S:APP}

We introduce the dynamic vaccination game and investigate about its dual formulation, in order to show the existence of
the Lagrange multipliers associated to the problem.

We regard a competition being played between population aggregations, possessing the
distinctness that we program to ensure a reply match in $[0,T]$. Our preamble is that at any time interval $t \in [0,T]$
there is a community of fixed proportion $N(t)$, however, we acquire $N(t) = N$ (it is comparable to affirm that communities that have the same  total measurement $N$ of parents of newborns are observed in $[0,T]$). At
each $t \in [0,T]$, the persons are observed with respect to
their respective perceptions of the relative risks of vaccine vs.
infection, scilicet $r_i$. Hence, we acquire a split of the
community of parents into agglomerations of proportions $\epsilon_i (t)$,
being $\displaystyle \sum^k_{i=1} \epsilon_i(t) = 1$, wherever
$\epsilon_i(t) \neq 0$ or 1 requesting precisely $k$ groups or units.
Unit $i$'s chance to become vaccinated at time t is indicated
by $P_i(t)$. Lastly, we assume that each unit's dimension
may change time to time and we formulate the proportions as functions
$\epsilon_i: [0,T] \to [0,1]$. Therefore each unit is thought to have
a vaccination strategy class of functions in the followin set
\begin{displaymath}
\mathbb{K}_i(t)= \{ P_i \in L^2([0,T],\mathbb{R}): \ P_i(t) \in
[0,1] \ {\rm almost\,\,everywhere \ in} \ [0,T] \}.
\end{displaymath}

Let us presume that all components of a unit $i$ have equivalent
 probability of significant morbidity due to vaccination,
indicated by $r^i_v(t)$, a.e. in $[0,T]$. Their chance
of becoming infected, given that a proportion $p$ of the community
is vaccinated, is indicated by $\pi^i_p(t)$, a.e. in $[0,T]$. In order to have a more general situation, we suppose that the perceived probability of becoming infected depends on the unit i's probability of getting vaccinated, namely $\pi^i_p=\pi^i_p(P(t))$.
Let us indicate their probability of considerable morbidity upon
infection by $r^i_{inf}(t)$, a.e. in $[0,T]$. The total
probability of experiencing an appreciable morbidity because
of not vaccinating is thus $\pi^i_p(P(t)) r^i_{inf}(t)$, a.e. in
$[0,T]$. At last, let us indicate by $r_i(t)=\displaystyle
\frac{r^i_v(t)}{r^i_{inf}(t)}$, a.e. in $[0,T]$, the relative
risk of vaccination versus infection of unit $i$. In addition, we postulate that the population consists of parents of newborns
and each group $i$ of parents is made up of individuals whose risk
assessments $r_i(t)$, a.e. in $[0,T]$, are moderately close.

%The perceived probability of significant morbidity due to
%vaccination is denoted by $r_v$.  The perceived probability of
%becoming infected given that a proportion $p$ of the population is
%vaccinated, is denoted by $\pi_p$, and the perceived probability of
%significant morbidity upon infection is denoted $r_{inf}$.  The
%overall perceived probability of experiencing significant morbidity
%because of not vaccinating is thus $\pi_p r_{inf}$.  We denote by
%$r_i:=\displaystyle{\frac{r^i_v}{r^i_{inf}}}$ the relative perceived
%risk of vaccination versus infection of group $i$. We study cases
%with $r_i\neq r_j,\text{ }\forall i,j\in\{1,2, \ldots ,k\}$,
%otherwise the problem reduces to the case of a population with $k-1$
%or less distinct groups.

As a consequence, we associate to each unit a payoff $u_i : [0,T]
\times L^2([0,T],\mathbb{R}) \to L^2([0,T ],\mathbb{R})$ given by
\begin{displaymath}
u_i(t, P(t)) =-r_i(t) P_i (t) - \pi^i_p(P(t))(1 - P_i(t)), \quad {\rm
a.e. \ in} \ [0, T],
\end{displaymath}
whilst the constraint set $\mathbb{K}$ turns into
\begin{displaymath}
\mathbb{K} = \left\{P \in L^2([0,T], \mathbb{R}^k): \ 0 \leq P_i(t)
\leq 1, \ \forall i \in \{1, \ldots, k\}, \ {\rm a.e. \ in} \ [0,T ]
\right\}.
\end{displaymath}
The vaccine coverage is now come to be a function $p : [0,T]
\to [0,1]$, estimated from the model as $\sum^k_{j=1}
\epsilon_j(t)P_j (t) = p(t)$, stated the action of the population
units. We express a time-dependent game of vaccination approach
over $\mathbb{K}$.

\begin{definition}\label{E:equilibrium}
An equilibrium state of a game is a vaccination strategy vector
$P^*\in L^2([0, T],\mathbb{R}^k)$,  that occurs
\begin{equation}\label{(6)}
P^*_i(t) \in \mathbb{K}_i(t): \ \ u_i(t,P^*(t)) \geq u_i(t,P_i (t),
\widehat{P}^*_i(t)), \quad \forall P_i(t) \in \mathbb{K}_i(t), \
\forall i=1, \ldots, k,
\end{equation}
wherever $\widehat{P}^*_i(t)= (P^*_1(t), \ldots, P^*_{i-1}(t), P^*_{i+1}(t),
\ldots, P^*_k(t))$.
\end{definition}

%Let us set
%\begin{displaymath}
%\nabla u(t, Q(t)) = \left( \frac{\partial u_1(t,Q(t))}{\partial
%P_1}, \ldots, \frac{\partial u_k(t,Q(t))}{\partial P_k}(t) \right).
%\end{displaymath}
Let us suppose that the following assumptions on the payoff functions $u_i$ and on
$\nabla u=\left(\dfrac{\partial u_1}{\partial P_1}, \ldots,
\dfrac{\partial u_k}{\partial P_k}\right)$ are satisfied:
\begin{enumerate}[(I)]
\renewcommand{\theenumi}{\roman{enumi}}
\item $u_i(t,P(t))$ is continuously differentiable a.e. in $\left[0,T\right]$,
\item $\nabla u$  is a Carath\'{e}odory
function such that
\begin{equation}\label{condcarath}
\exists h \in L^2(\left[0,T\right]):\left\|\nabla
u(t,P(t))\right\|_{mn} \leq h(t)\left\|P(t)\right\|_{mn},\; \mbox{
a.e. in} \left[0,T\right], \, \forall P \in L^2([0,T],
\mathbb{R}^k),
\end{equation}
\item $u_i(t,P(t))$ is pseudoconcave with respect to $P_i$, $i=1,\ldots,k,$ that is, for a.e. in $[0,T]$:
\begin{displaymath}
\langle \frac{\partial u_i}{\partial P_i}(t,P_1, \ldots, P_i, \ldots
P_k), P_i - Q_i \rangle \geq 0
\end{displaymath}
\begin{displaymath}
\Longrightarrow \ u_i(t, P_1, \ldots, P_i, \ldots P_k) \geq u (t,
P_1, \ldots, Q_i, \ldots P_k).
\end{displaymath}
\end{enumerate}

We underline that the time-dependent vaccination equilibrium is
characterized by means of an evolutionary variational inequality, as we put in evidence in
the successive result (see \cite{BarbaCoj3}, Theorem 3.2):

\begin{theorem}
	Let us assume that assumptions (i), (ii) and (iii) are satisfied.
A time-dependent vaccination equilibrium is reached, in accord with Definition
$\ref{E:equilibrium}$, if and only if it satisfies the evolutionary variational inequality
\begin{equation}\label{(7)}
\int^T_0 \langle - \nabla u(t, Q(t)), P(t) - Q(t) \rangle dt \geq 0,
\quad \forallP \in K.
\end{equation}
%where 
%\begin{displaymath}
%\nabla u(t, Q(t)) = \left( \frac{\partial u_1(t,Q(t))}{\partial
%P_1}, \ldots, \frac{\partial u_k(t,Q(t))}{\partial P_k}(t) \right).
%\end{displaymath}
\end{theorem}

We keep in mind that, in the Hilbert space $L^2([0, T],\mathbb{R}^k)$,
\begin{displaymath}
\ll \phi, y \gg= \int^T_0 \langle \phi(t), y(t) \rangle dt
\end{displaymath}
is the duality mapping, being $\phi \in (L^2([0, T],\mathbb{R}^k))^*
= L^2([0, T],\mathbb{R}^k)$ and $y \in L^2([0, T],\mathbb{R}^k)$.

Our purpose is to acquire the existence of solution to \eqref{(7)}. To this aim we take into account some preliminaries (see \cite{MauRac}).

\begin{definition}
If $\nabla u= \left(\frac{\partial u_1}{\partial P_1}, \ldots,
\frac{\partial u_k}{\partial P_k}\right)$ is as before. We say that
\begin{itemize}
\item the mapping $- \nabla u$ is pseudomonotone in the
sense of Karamardian (briefly K-pseudomonotone) iff for every
$P, Q\in\mathbb{K}$
$$\ll - \nabla u(Q),P-Q
\gg \geq 0 \Rightarrow
\ll - \nabla u(P) ,P-Q
\gg \geq 0;$$
\item the mapping $- \nabla u$ is pseudomonotone in the sense of Brezis (briefly B-pseudomonotone) iff
\begin{enumerate}
\item for each sequence $\{P_n\}$ weakly converging to u (in short $P_n\rightharpoonup P$)
in $\mathbb{K}$ and such that $\lim\sup_n
%\left  modifica 10(1)2014
\ll
- \nabla u(P_n),P_n-Q
%\right modifica 10(1)2014
\gg\leq 0, $
it ensures that
$$\liminf_{n \to + \infty}
%\left modifica 10(1)2014
\ll - \nabla u(P_n),P_n-Q
%\right modifica 10(1)2014
\gg\geq
%\left modifica 10(1)2014
\ll
- \nabla u(P),P-Q
%\right modifica 10(1)2014
\gg,\quad\forall Q\in\mathbb{K};$$
\item for each $Q\in\mathbb{K}$ the function $P \mapsto
%\left modifica 10(1)2014
\ll - \nabla u(P),P-Q
%\right modifica 10(1)2014
\gg$
is lower bounded on the bounded subset of $\mathbb{K};$
\end{enumerate}
\item the mapping $- \nabla u$ is lower hemicontinuous
along line segments, iff the function $\xi \mapsto
%\left modifica 10(1)2014
\ll
- \nabla u(\xi),P-Q
%\right modifica 10(1)2014
\gg$ is lower semicontinuous for every
$P,Q\in\mathbb{K}$ on the line segments $\left[P,Q\right];$
\item the mapping $-\nabla u$ is hemicontinuous in the
sense of Fan (briefly F-hemicontinuous) iff the
function $P \mapsto
%\left modifica 10(1)2014
\ll - \nabla u(P),P-Q
%\right modifica 10(1)2014
\gg$ is weakly
lower semicontinuous on $\mathbb{K}$, for every $Q\in\mathbb{K}$.
\end{itemize}
\end{definition}

Finally, we can explain the existence result that can be achived making use of some arguments contained in \cite{MauRac}.

\begin{theorem}
	Let us assume that assumptions (i), (ii) and (iii) are satisfied.
If $- \nabla u$ is B-pseudomonotone,  or F-hemicontinuous or $- \nabla u$
is K-pseudomonotone,
then $(\ref{(7)})$
admits a solution.
\end{theorem}
%\begin{proof}
%Let us note that $\mathbb{K}$ is clearly a nonempty, closed, convex
%and bounded subset of $L^2(\left[0,T\right],\mathbb{R}^{mn})$ and,
%consequently, it is a weakly compact subset of
%$L^2(\left[0,T\right],\mathbb{R}^{mn})$.% because
%%$L^2(\left[0,T\right],\mathbb{R}^{mn})$ is a reflexive Banach space.
%Then, the claim is achieved by applying Theorems 2.6 and 2.7 and
%Corollary 3.7 in \cite{MauRac}.
%\end{proof}

Let us note that (\ref{condcarath}) of (ii) guarantees the lower semicontinuity along line segments of
$- \nabla u$. As a consequence, assuming that
$- \nabla u$ is a K-pseudomonotone operator, then the existence
of a solution is ensured without entering supplementary assumptions besides
(\ref{condcarath}).

%%%%%%%%%%%%%%%%%%%%%%%%%%%%%%%%%%%%%%%%%%%%%%%%%%%%%%%%%%%%%%%%%%%%

Now we apply the infinite dimensional duality theory to the dynamic
vaccination game. For this intent, let $%P
Q \in \mathbb{K}$ be a solution of $(\ref{(7)})$ and let us set up
\begin{displaymath}
\psi (P) = \ll - \nabla u(Q), P - Q \gg, \quad \forall P \in
\mathbb{K},
\end{displaymath}
that is
\begin{displaymath}
\psi (P) = \int_0^1 - \sum_{i=1}^k \frac{\partial
u_i(t,Q(t))}{\partial P_i} (P_i(t) - Q_i(t)) dt, \quad \forall P \in
\mathbb{K}.
\end{displaymath}
Moreover, let us note
\begin{displaymath}
\psi (P) \geq 0 \quad \forall P \in \mathbb{K}
\end{displaymath}
and
\begin{equation}\label{E:star}
\min_{P \in \mathbb{K}} \psi(P) = \psi(Q)=0.
\end{equation}
We combine $(\ref{(7)})$ and the subsequent Lagrange functional:
\begin{displaymath}
\mathcal{L}(P, \alpha, \beta) = \psi(P) - \ll \alpha, P \gg + \ll
\beta, P - 1 \gg,
\end{displaymath}
for every $P \in L^2([0,T], \mathbb{R}^{k}), \ (\alpha, \beta) \in
C^*$, being
\begin{eqnarray*}
C^* = \Big\{ \big( \alpha, \beta \big) \in L^2([0,T],
\mathbb{R}^{k}) \times L^2([0,T], \mathbb{R}^{k}): \ \ \alpha(t)
\geq 0, \ \beta(t) \geq 0, \ \ && \\ {\rm a. e. \ in} \ [0, T] \Big\}
&&
\end{eqnarray*}
the dual cone of the ordering cone $C$ of $L^2([0,T],
\mathbb{R}^{k}) \times L^2([0,T], \mathbb{R}^{k})$. We point out that, in our
case, $C=C^*$. %and the function $\psi(x)$, $g_1(t)=
%\underline{x}(t) - x(t)$, $g_2(t)=x(t)- \overline{x}(t)$ are linear.
We are ready to demonstrate that our problem complies {\em Assumption S}.
To purchase this goal, we need to state a preparatory lemma.

\begin{lemma}
Let us assume that assumptions (i), (ii) and (iii) are satisfied.
Let $P \in \mathbb{K}$ be a solution of $(\ref{(7)})$ and fix, for $i=1,2, \ldots, m, $
\begin{displaymath}
E^{i}_- = \{ t \in [0, T]: \ Q_{i}(t) = 0 \}, \ \ E^{i}_0 = \{ t \in
[0, T]: \ 0 < Q_{i}(t) < 1 \},
\end{displaymath}
\begin{displaymath}
E^{i}_+ = \{ t \in [0, T]: \ Q_{i}(t) = 1 \}.
\end{displaymath}
Then, we reach
\begin{displaymath}
\frac{\partial u_i(t,Q(t))}{\partial P_{i}} \leq 0, \ \ {\rm a.e. \ in} \
E_-^{i}, \ \quad \frac{\partial u_i(t,Q(t))}{\partial P_{i}} = 0, \ \
{\rm a.e. \ in} \ E_0^{i},
\end{displaymath}
\begin{displaymath}
\frac{\partial u_i(t,Q(t))}{\partial P_{i}} \geq 0, \ \ {\rm a.e. \ in} \
E_+^{i}.
\end{displaymath}
\end{lemma}
\begin{proof}
In order to advantage the reader, we report the used technique and some details (see
also \cite{BarbaMau}, Lemma 4.7). We emphasize that
\begin{eqnarray*}
\ll - \nabla u (Q), P- Q \gg &=& - \int_0^T
\sum_{i=1}^k  \frac{\partial u_i(t, Q(t))}{\partial P_{i}} (P_{i}(t) - Q_{i}(t)) dt \\
&=& - \sum_{i=1}^k \int_{E^{i}_-} \frac{\partial
u(t, Q(t))}{\partial P_{i}} P_{i}(t) dt \\
&& - \sum_{i=1}^k \int_{E^{i}_0} \frac{\partial
u_i(t, Q(t))}{\partial P_{i}} (P_{i}(t) - Q_{i}(t)) dt \\
&& - \sum_{i=1}^k \int_{E^{i}_+} \frac{\partial u_i(t,
Q(t))}{\partial P_{i}} (P_{i}(t) - 1) dt \geq 0.
\end{eqnarray*}
Let us presuppose that $P_{l}(t) = Q_{l}(t)$ for $l \neq i$, we get,
for every $0 \leq P_{i}(t) \leq 1$ and every $i=1,2, \ldots, m$,
\begin{eqnarray}\label{E:dual1}
\ll - \nabla u (Q), P- Q \gg &=& - \int_{E^{i}_-} \frac{\partial
u_i(t, Q(t))}{\partial P_{i}} P_{i}(t) dt \nonumber
\\ && - \int_{E^{i}_0} \frac{\partial u_i(t, Q(t))}{\partial P_{i}}
(P_{i}(t) - Q_{i}(t)) dt \\ && - \int_{E^{i}_+} \frac{\partial
u_i(t, Q(t))}{\partial P_{i}} (P_{i}(t) - 1) dt \geq 0. \nonumber
\end{eqnarray}
We pick out that
\begin{displaymath}
P_{i}
\begin{cases}
\geq 0 \quad {\rm in} \ E_-^{i} \\
= Q_{i} \quad {\rm in} \ E_0^{i} \\
= Q_{i} \quad {\rm in} \ E_+^{i}
\end{cases},
\end{displaymath}
so that $(\ref{E:dual1})$ could be written as follows:
\begin{equation}\label{E:dual2}
- \int_{E^{i}_-} \frac{\partial u_i(t, Q(t))}{\partial P_{i}}
P_{i}(t)  dt \geq 0.
\end{equation}
Our ``finishing line"  is to show that  $\dfrac{\partial u_i(t, Q(t))}{\partial P_i} \leq
0$, a.e. in $E^i_-$. To this purpose, if there exists a subset $F$ of
$E_-^{i}$, with $m(F)>0$, such that $- \dfrac{\partial
u_i(t,Q(t))}{\partial P_{i}} < 0$ in $F$, picking
\begin{displaymath}
P_{i}
\begin{cases}
= 0 \quad {\rm in} \ E_-^{i} \setminus F \\
> 0 \quad {\rm in} \ F
\end{cases},
\end{displaymath}
it results
\begin{displaymath}
- \int_F \frac{\partial u_i(t, Q(t))}{\partial P_{i}} P_{i}(t) dt <
0,
\end{displaymath}
that contradicts $(\ref{E:dual2})$. The product effect is that
\begin{displaymath}
\frac{\partial u_i(t, Q(t))}{\partial P_{i}} \leq 0, \quad {\rm in}
\ E_-^{i}.
\end{displaymath}
Using similar arguments, we reach
\begin{displaymath}
\frac{\partial u_i(t, Q(t))}{\partial P_{i}} = 0, \quad {\rm in} \
E_0^{i}.
\end{displaymath}
In $E_+^{i}$, selecting
\begin{displaymath}
P_{i}
\begin{cases}
= 0 \quad {\rm in} \ E_-^{i} \\
= Q_{i} \quad {\rm in} \ E_0^{i} \\
\leq 1 \quad {\rm in} \ E_+^{i}
\end{cases},
\end{displaymath}
we gain
\begin{displaymath}
- \int_{E_+^{i}} \frac{\partial u_i(t, Q(t))}{\partial P_{i}}
(P_{i}(t) - 1) dt \geq 0.
\end{displaymath}
When $P_{i} - 1 \leq 0$ in $E_+^{i}$ then we prove that $\dfrac{\partial u_i(t,
Q(t))}{\partial P_{i}} \geq 0$. As a matter of fact, if there exists $F
\subseteq E_+^{i}$, with $m(F) >0$, such that $\dfrac{\partial u_i(t,
Q(t))}{\partial P_{i}} < 0$ in $F$, considering
\begin{displaymath}
P_{i}
\begin{cases}
= 1 \quad {\rm in} \ E_+^{i} \setminus F \\
< 1 \quad {\rm in} \ F
\end{cases},
\end{displaymath}
it follows
\begin{displaymath}
- \int_F \frac{\partial u_i(t, Q(t))}{\partial P_{i}} P_{i}(t) dt <
0.
\end{displaymath}
As an effect,
\begin{displaymath}
\frac{\partial u_i(t, Q(t))}{\partial P_{i}} \geq 0, \quad {\rm in}
\ E_+^{i}.
\end{displaymath}
\end{proof}

Now we are ready to announce the main result of the article.

\begin{theorem}
Let us assume that assumptions (i). (ii) and (iii) are satisfied.
Problem $(\ref{E:star})$ satisfies {\em Assumption S} at the minimal
point $P \in \mathbb{K}$.
\end{theorem}
\begin{proof}
We remember that
\begin{displaymath}
\psi (%Q
P) = \ll - \nabla u(Q), P - Q \gg, \quad \forall
%Q
P \in
\mathbb{K}.
\end{displaymath}
To reach {\em Assumption S}, we demonstrate that if $(l,
\theta_{L^2([0,T], \mathbb{R}^{k})}, \theta_{L^2([0,T],
\mathbb{R}^{k})})$ belongs to $T_{\widetilde{M}}(0,
\theta_{L^2([0,T], \mathbb{R}^{k})}, \theta_{L^2([0,T],
\mathbb{R}^{k})})$, specifically if
\begin{eqnarray*}
l &=& \lim_{n \to + \infty} \lambda_n
(
 %\phi
\psi(%Q
P_n) -
%\phi
\psi(%P
Q) + \alpha_n
), \\
\theta_{L^2([0,T], \mathbb{R}^{k})} &=& \lim_{n \to + \infty}
\lambda_n
(-%Q
P_n + R_n), \\
\theta_{L^2([0,T], \mathbb{R}^{k})} &=& \lim_{n \to + \infty} %(1 -Q_n + S_n)
\lambda_n (P_n -1 + S_n),
\end{eqnarray*}
with $\lambda_n > 0$, $0= \lim_{n \to + \infty}
(\psi(P_n) - \psi(Q) + \alpha_n)$, $\theta_{L^2([0,T], \mathbb{R}^{k})}= \lim_{n
\to + \infty} %(- Q + R_n)
(-P_n + R_n)
$,
$\theta_{L^2([0,T], \mathbb{R}^{k})} =
\lim_{n \to + \infty}
(
%Q
P_n(t) - 1 + S_n)$, $
%Q
P_n \in L^2([0,T],
\mathbb{R}^{k}) \setminus \mathbb{K}$, $\alpha_n \geq 0$, $R_n \in
C$, then $l \geq 0$.

We have
\begin{eqnarray*}
l &=& \lim_{n \to + \infty} \lambda_n \Big( - \sum_{i=1}^k \int_0^T
\frac{\partial u_i(t, Q(t))}{\partial P_{i}} (P_{i}^n(t) - Q_{i}(t))
dt
+ \alpha_n \Big) \\
&\geq& \lim_{n \to + \infty} \lambda_n \Big( - \sum_{i=1}^k
\int_{E_-^{i}} \frac{\partial u_i(t, Q(t))}{\partial
P_{i}} P_{i}^n(t) dt \\
&&- \sum_{i=1}^k \int_{E_0^{i}} \frac{\partial u_i(t,
Q(t))}{\partial
P_{i}} (P_{i}^n(t) - Q_{i}(t)) dt \\
&& - \sum_{i=1}^k \int_{E_+^{i}} \frac{\partial u_i(t,
Q(t))}{\partial P_{i}} (P_{i}^n(t) - 1) dt \Big).
\end{eqnarray*}
Furthermore we notice that
\begin{eqnarray*}
&&\lim_{n \to + \infty} \lambda_n \Bigg[ - \sum_{i=1}^k
\int_{E_-^{i}} \frac{\partial u_i(t, Q(t))}{\partial P_{i}}
P_{i}^n(t) dt \Bigg]
\\ && \quad = \lim_{n \to + \infty} \lambda_n \Bigg[ \sum_{i=1}^k
\Bigg( - \int_{E_-^{i}} \frac{\partial u_i(t, Q(t))}{\partial
P_{i}} (P_{i}^n(t) - R_{i}^n(t)) dt \\
&& \quad \ \ \ \  - \int_{E_-^{i}} \frac{\partial u_i(t,
Q(t))}{\partial P_{i}} R_{i}^n(t) dt \Bigg) \Bigg] \geq 0,
\end{eqnarray*}
because of $\lim_{n \to + \infty} \lambda_n (
%Q
P_{i}^n - R_{i}^n)= \theta_{L^2([0,T], \mathbb{R})}, \ {\rm in} \
L^2([0,T], \mathbb{R}), \quad R_{i}^n \geq 0, \quad \lambda_n \geq
0$, $\dfrac{\partial u_i(t,Q(t))}{\partial P_{i}} \leq 0$ in
$E^{i}_-$.

In addition, we observe that
\begin{displaymath}
\lim_{n \to + \infty} \lambda_n \sum_{i=1}^k \int_{E_0^i}
\frac{\partial u_i(t, Q(t))}{\partial P_{i}} (P_{i}^n(t) - Q_{i}(t))
dt = 0,
\end{displaymath}
being $\frac{\partial u_i(t,Q(t))}{\partial P_{i}} = 0$ in $E^i_0$.

Since $\lim_{n \to + \infty} \lambda_n (Q_{i}^n + S_{i}^n - 1)=0, \
{\rm in} \ L^2([0,T], \mathbb{R}), \quad \lambda_n \geq 0, \quad
S_{i}^n \geq 0$, $\dfrac{\partial u_i(t,Q(t))}{\partial P_{i}}
\geq 0$, in $E^{i}_+$,
it results
\begin{eqnarray*}
&& \lim_{n \to + \infty} \lambda_n \sum_{i=1}^k \int_{E_+^{i}} -
\frac{\partial u_i(t, Q(t))}{\partial P_{i}} (P_{i}^n(t) - 1) dt
\\ && \quad = \lim_{n \to + \infty} \lambda_n \sum_{i=1}^k
\int_{E_+^{i}} - \frac{\partial u_i(t, Q(t))}{\partial P_{i}}
(P_{i}^n(t) + S_{i}^n(t) - 1 - S_{i}^n(t)) dt
\\ && \quad = \lim_{n \to + \infty} \sum_{i=1}^k \int_{E_+^{i}} - \frac{\partial u_i(t, Q(t))}{\partial P_{i}}
\lambda_n (P_{i}^n(t) + S_{i}^n(t) - 1) dt \\
&& \quad \ \ + \lim_{n \to + \infty} \lambda_n \sum_{i=1}^k
\int_{E_+^{i}} \frac{\partial u_i(t, Q(t))}{\partial P_{i}}
S_{i}^n(t) dt \geq 0.
\end{eqnarray*}
Hence, the claim is completely achived.
\end{proof}

Let us now prove a necessary and sufficient condition concerning solutions to $(\ref{(7)})$.
%%%%%%%%%%%%%%%%%%%%%%%%%%%%%%%%%%%%%%%%%%%%%%%%%%%%%%%%%%%%

\begin{theorem}\label{T:trenove}
	Let us assume that assumptions (i), (ii) and (iii) are satisfied.
$P \in \mathbb{K}$ is a solution to
$(\ref{(7)})$ if and only if there exist $\alpha^*, \beta^* \in
L^2([0,T], \mathbb{R}^{k})$ such that:
\begin{itemize}
\item[(I)] $\alpha^*(t), \beta^*(t) \geq 0$, a.e. in $[0,T]$;
\item[(II)] $\alpha^*(t) Q(t) =0$, a.e. in $[0,T]$,
\\$\beta^*(t) (Q(t) - 1) =0$, a.e. in $[0,T]$;
\item[(III)] $ - \nabla u(t,Q(t)) + \beta^*(t) = \alpha^*(t)$, a.e. in $[0,T]$.
\end{itemize}
\end{theorem}
\begin{proof}
Making use of Theorem $\ref{thm2}$ there exists $(\alpha^*, \beta^*) \in C^*$
such that $(Q, \alpha^*, \beta^*)$ is a saddle point of the
functional $\mathcal{L}$, namely
\begin{equation}\label{E:undici}
\mathcal{L}(Q, \alpha, \beta) \leq \mathcal{L}(Q, \alpha^*, \beta^*)
\leq \mathcal{L}(P, \alpha^*, \beta^*), \quad \forall (\alpha,
\beta) \in C^*, \ P \in L^2([0,T], \mathbb{R}^{k}),
\end{equation}
and, in addition,
\begin{eqnarray}\label{E:dodici}
\ll \alpha^*, Q \gg &=& 0, \\
\ll \beta^*, Q - 1 \gg &=& 0.
\end{eqnarray}
Since $\alpha, \beta \geq 0$, $Q \geq 0$, $Q - 1 \leq 0$, a.e. in $[0,T]$,
$(\ref{E:dodici})$ implies
\begin{eqnarray*}
\alpha^*(t) Q(t) &=& 0, \quad {\rm a.e.\ in}
\ [0,T], \\
\beta^*(t) (Q(t) - 1) &=& 0, \quad {\rm a.e. \ in} \ [0,T].
\end{eqnarray*}
Thanks to $(\ref{E:undici})$ we obtain, for every $P \in L^2([0,T],
\mathbb{R}^{k})$,
\begin{eqnarray*}
\mathcal{L}(P, \alpha^*, \beta^*) &=& \ll - \nabla u(Q), P - Q \gg
- \ll \alpha^*, P \gg + \ll \beta^*, P - 1 \gg \\
&\geq& 0 = \mathcal{L}(Q, \alpha^*, \beta^*).
\end{eqnarray*}
Bearing in mind $(\ref{E:dodici})$, we reach
\begin{displaymath}
\ll - \nabla u(Q)- \alpha^* + \beta^*, P - Q \gg \geq 0, \quad
\forall P \in L^2([0,T], \mathbb{R}^{k}).
\end{displaymath}
Let us fix
\begin{displaymath}
P^1= Q + \varepsilon, \quad P^2 = Q - \varepsilon, \quad \forall
\varepsilon \in C^{\infty}_0([0,T], \mathbb{R}^{k}),
\end{displaymath}
we acquire, for every $\varepsilon \in C^{\infty}_0([0,T],
\mathbb{R}^{k})$,
\begin{eqnarray}\label{E:epsilon}
\mathcal{L}(
%Q
P^1, \alpha^*, \beta^*) &=& - \ll - \nabla u(Q)-
\alpha^* + \beta^*, \varepsilon \gg \geq 0, \\
\mathcal{L}(
%Q
P^2, \alpha^*, \beta^*) &=& \ll - \nabla u(Q)- \alpha^*
+ \beta^*, \varepsilon \gg \geq 0.
\end{eqnarray}
As a consequence, by $(\ref{E:epsilon})$, we derive, for every
$\varepsilon \in C_0^{\infty}([0,T])$:
\begin{displaymath}
\ll - \nabla u(Q)- \alpha^* + \beta^*, \varepsilon \gg = 0,
\end{displaymath}
precisely, we gain
\begin{equation}\label{E:risultato}
- \nabla u(t,Q(t))- \alpha^*(t) + \beta^*(t) =0, \quad {\rm a.e. \
in} \ [0,T].
\end{equation}

Vice versa, if there exists $Q \in \mathbb{K}$, $\alpha^* \in
L^2([0,T], \mathbb{R}^k)$ and $\beta^* \in L^2([0,T], \mathbb{R})$
that satisfy assumptions (I), (II) and (III), it arises that $(Q,
\alpha^*, \beta^*)$ is a saddle point of $\mathcal{L}$. Then, reminding  Theorem $\ref{thm2}$, we affirm that  $Q$ is a solution to $(\ref{(7)})$.
\end{proof}

%%%%%%%%%%%%%%%%%%%%%%%%%%%%%%%%%%%%%%%%%%%%%%%%%%%%%%%
Recalling the above proved Theorem \ref{T:trenove}, we can emphasize the
importance of the Lagrange multipliers $\alpha^*$ and $\beta^*$ on
the understanding and the management of the vaccination game
problem. Indeed, at a fixed time $t \in [0,T]$, it emerges that
\begin{enumerate}[(a)]
\item if $\alpha^*(t)>0$ then, using (II), we have
$Q(t)=0$, namely the group $i$'s probability of getting vaccinated
is null;
\item if $Q(t)>0$ then, bearing in mind (II), $\alpha^*(t)=0$
and, making use of (III), it results $\nabla u_i(t,Q(t)) =
\beta^*(t),$ namely $\beta^*(t)$ is equal to the marginal payoff;
\item if $\beta^*(t)>0$ then, from (II), we obtain
$Q(t)=1$, namely the group $i$'s probability of getting vaccinated
is maximum;
\item if $Q(t)<1$ then, from (II), $\beta^*(t)=0$
and, taking into account (III), we obtain $- \nabla u_i(t,Q(t))=
\alpha^*(t),$ namely $\beta^*(t)$ is equal to the opposite of the
marginal payoff.
\end{enumerate}

\section*{Acknowledgment}
We wish to warmly thank Professor A. Maugeri for the proficuous discussions.

\end{document}